\documentclass[11pt]{article}
\usepackage[letterpaper,textwidth=6.25in,textheight=8.9in,centering]{geometry}
\usepackage{lmodern}
\usepackage[T1]{fontenc}
\input{glyphtounicode}
\usepackage{amsmath,amssymb,amsthm}
\usepackage{microtype}
\usepackage{enumitem}
\setlist[enumerate]{beginpenalty=10000}
\usepackage{xurl}
\usepackage[hidelinks]{hyperref}
\hypersetup{pdftitle={A lattice order on the complex field},
  pdfauthor={Hanzhen Bao; Yuehui Zhang},
  pdfkeywords={lattice-ordered field, positive cone, Birkhoff--Pierce problem, valuation},
  bookmarksnumbered=true,pdfdisplaydoctitle=true}
\newtheorem*{mainthm}{Main Theorem}

\theoremstyle{plain}
\newtheorem{theorem}{Theorem}[section]
\newtheorem{proposition}[theorem]{Proposition}
\newtheorem{lemma}[theorem]{Lemma}
\newtheorem{corollary}[theorem]{Corollary}
\newtheorem{question}[theorem]{Question}
\theoremstyle{definition}
\newtheorem{definition}[theorem]{Definition}
\newtheorem{example}[theorem]{Example}
\theoremstyle{remark}
\newtheorem{remark}[theorem]{Remark}

\newcommand{\Q}{\mathbb{Q}}
\newcommand{\R}{\mathbb{R}}
\newcommand{\C}{\mathbb{C}}
\newcommand{\Z}{\mathbb{Z}}
\newcommand{\N}{\mathbb{N}}
\DeclareMathOperator{\Sat}{Sat}
\newcommand{\doi}[1]{doi:\ \href{https://doi.org/#1}{\nolinkurl{#1}}}

\makeatletter
\renewcommand{\@seccntformat}[1]{\csname the#1\endcsname.\quad}
\makeatother
\title{\bfseries A lattice order on the complex field}
\author{Hanzhen Bao\qquad Yuehui Zhang}
\date{}
\begin{document}
\maketitle
\begin{abstract}
Birkhoff and Pierce asked in 1956 whether the field $\C$ of complex numbers admits a lattice order compatible with addition and multiplication. We answer the question affirmatively. Using a valuation with value group $\Q$, we construct a positive cone $P$ by a transfinite recursion of countable ordinal length. At each stage we extend a lattice order on a subspace while preserving its lattice operations in the whole field and maintaining multiplicative closure of the ambient positive cone. These subspaces eventually exhaust $\C$. The construction depends on choices of a valuation extension, compatible rational powers and vector-space bases. The cone $P$ contains $1$ and satisfies
$\{x\in P\setminus\{0\}:x^{-1}\in P\}=\Q_{>0}$.
The induced order is not total: $i$ is incomparable with $0$. For a chosen transcendental element $t$ and $i^2=-1$, one may prescribe
$0\vee i=(t+i)/2$ and $0\vee(-i)=(t-i)/2$.
Varying the first extension yields exactly $2^{\mathbf c}$ distinct lattice positive cones on $\C$, where $\mathbf c=|\C|$. The count concerns cones as subsets of $\C$.
\end{abstract}
{\small\noindent\textbf{Key words and phrases.} Lattice-ordered field; positive cone; Birkhoff--Pierce
problem; valuation; transcendence basis.\par}
{\small\noindent\textbf{2020 Mathematics Subject Classification.} Primary 12J15; Secondary 06F25, 12J20.\par}
\section{Introduction}\label{sec:intro}

A \emph{ring order} on a field $L$ is a translation-invariant partial order for which a product of nonnegative elements is nonnegative. It is determined by a subset $P\subseteq L$ satisfying
\[
0\in P,\qquad P+P\subseteq P,\qquad PP\subseteq P,\qquad P\cap(-P)=\{0\}.
\]
Such an order is \emph{directed} if $P-P=L$. It is a \emph{lattice order} if every pair has a least upper bound and a greatest lower bound; in that case it is automatically directed, and $(L,\le_P)$ is a \emph{lattice-ordered field}.

\subsection*{The problem}

\begin{quote}\itshape
\textup{Birkhoff--Pierce problem (1956).} Does $\C$ admit a lattice order compatible with its ring operations?
\end{quote}

Birkhoff and Pierce posed this question in their foundational paper on lattice-ordered rings \cite{BP}. They excluded lattice orders compatible with the usual $\R$-algebra structure of $\C$, leaving the ring-order problem open. Fuchs placed such questions in the broader study of partially ordered algebraic systems \cite{Fuchs}. Conrad and Dauns constructed non-total lattice-ordered fields and established a Hahn-type embedding theorem under additional structural hypotheses \cite{CD}. Wilson subsequently constructed $2^{\mathbf c}$ non-total lattice orders on $\R$, where $\mathbf c=|\R|$ \cite{W}.

Two restrictions are relevant to the complex case. First, $\C$ cannot be totally ordered as a field: totality would make every square nonnegative, contrary to $i^2=-1$. Second, Ma's obstruction excludes directed ring orders on $\overline{\Q}$ \cite{Ma}. Thus a directed order on $\C$ cannot restrict to a directed order on $\overline{\Q}$; the restriction in our construction is not directed.

Schwartz and Yang established the existence of directed ring orders on $\C$ \cite{SY} and later gave a general existence criterion \cite{SY23}. The construction below preserves lattice operations throughout the extension process. The background and related existence questions are discussed in \cite{Hen,Y}.

\subsection*{Cones over a fixed coefficient field}

Another line of work fixes a non-Archimedean totally ordered field $F$ and studies positive cones on $F(i)$ satisfying $F_{>0}P\subseteq P$. This is compatibility with the prescribed positive scalar action. If $1\in P$, this scalar condition is equivalent to $F_{\ge0}\subseteq P$, and the induced order on $F$ is then the prescribed total order. The scalar condition alone, without $1\in P$, does not assert this restriction property.

Ma, Wu and Zhang constructed such cones \cite{MWZ17} and classified the cases $1>0$ and $1\not>0$ by admissible semigroups and special convex sets, respectively \cite{MWZ18,MWZ19}. Xu and Zhang unified these descriptions by doubly convex sets and extended the analysis to quaternion algebras \cite{XZ20}. The non-lattice conclusions concern the directed $F$-algebra orders in that prescribed scalar-compatibility framework.

Wang, Yuan, Zhang and Zhu recently gave a construction of directed cones on $\C$ using integral closures of local rings \cite{WYZZ}. Their cones are closed under inversion of nonzero positive elements and are not lattice cones. Section~\ref{sub:previous} compares their restriction and inversion properties with those of our cones.

\subsection*{The result}

\begin{mainthm}\label{mainthm}
The field $\C$ carries a lattice positive cone $P$. Moreover $P$ can be chosen so that
\[
1\in P,\qquad \Q_{>0}\subseteq P,\qquad P\cap\Q=\Q_{\ge0},\qquad
\{x\in P\setminus\{0\}:\ x^{-1}\in P\}=\Q_{>0},
\]
and so that, for $i^{2}=-1$ and $p=\frac12(t+i)$, $q=\frac12(t-i)$, where $t$ is a fixed element
of a transcendence basis of $\C$ over $\Q$,
\[
0\vee i=p,\qquad 0\vee(-i)=q,\qquad p\wedge q=0.
\]
There are $2^{\mathbf c}$ distinct cones satisfying these conditions, where $\mathbf c=|\C|$. For each such cone, $\Q$ is the only subfield of $\C$ on which the induced order is total.
\end{mainthm}

\noindent
Incomparability of $i$ with $0$ is already determined in a three-dimensional subspace. The existence assertion follows from Theorem~\ref{thm:C}; Corollary~\ref{cor:many} gives the stated number of cones, and Corollary~\ref{cor:total} gives the assertion about totally ordered subfields.

The main theorem also disproves the conjecture in \cite{WYZZ} that $\C$ has no compatible lattice order. Under the support condition used here, inversion closure is incompatible with directedness (Remark~\ref{rem:inversion}).

In a recent preprint, He \cite{He26} presents a construction of a compatible lattice order on the complex field, using a coordinate order on a countable algebraically closed subfield followed by an ultrapower construction and transport through a field isomorphism.

The construction and the proofs of the results stated in our main theorem were completed before the authors became aware of He's preprint. They were developed independently and do not use any results or arguments from that work. The discussion of He's preprint, including the comparison in Section~\ref{sub:previous}, was added subsequently.

Our construction proceeds directly on the fixed field and yields $2^{\mathbf c}$ distinct lattice positive cones whose only nonzero positive elements with positive inverses are the positive rational numbers. The resulting orders are not isomorphic as ordered fields to those arising from He's ultrapower construction; see Section~\ref{sub:previous}.

\subsection*{How the cone is built}

Write $L=\C$. Fix a transcendence basis $T$ of $L$ over $\Q$, choose $t\in T$, and put $F_0=\Q(T\setminus\{t\})$. The valuation at infinity of $F_0(t)$ extends to a valuation $v$ on $L$ with value group $\Q$. Choose compatible rational powers $t^r$, $r\in\Q$. For the height $d=-v$, use the selected layers
\[
H_s=\{ct^s+u:c\in\Q_{>0},\ d(u)<s\}.
\]
Each $H_s$ consists of elements whose class in $L_{\le s}/L_{<s}$ is a positive rational multiple of the class of $t^s$; no order on the whole residue field is imposed.

The cone is an increasing union of positive cones $P_\alpha$ on the whole field. At stage $\alpha$, a lattice cone $A_\alpha=P_\alpha\cap U_\alpha$ is defined on a subspace $U_\alpha\subseteq L_{\le a_\alpha}$. At every noninitial stage it also satisfies $L_{<a_\alpha}\subseteq U_\alpha$. The subspace need not be all of $L_{\le a_\alpha}$. Outside it, $P_\alpha$ consists of the layers with indices $s>a_\alpha$ in a well-ordered additive monoid $S_\alpha$, generated by $D_\alpha\subseteq\Q\cap[1,a_\alpha]$.

At a successor step choose a rational height $h$ between the current cut and the next element of $S_\alpha$. A basis $(x_j+U_\alpha)$ of $L_{<h}/U_\alpha$ supplies new coordinate directions $t^h$ and $t^h+x_j$. At a finite limit cut $a$, the previous cones have an increasing union. If $a$ belongs to the union height monoid, the already positive layer $H_a$ supplies a lexicographic extension along $\Q t^a$; otherwise the union of the earlier subspaces suffices. In both cases the new whole-field cone is exactly the union of the earlier cones.

The cuts increase strictly, and successor steps and finite-cut limit steps are always available. Since distinct successor stages have distinct rational cuts, the recursion stops at a countable limit ordinal at which the cuts become unbounded (Theorem~\ref{thm:main}). The subspaces then exhaust $\C$. Coordinate and lexicographic extensions preserve the old reduced decompositions locally, and globalisation at each new stage preserves them in the whole field. Every element therefore acquires a positive part.

The construction depends on noncanonical choices and does not provide an algorithm for computing the positive part of an arbitrary complex number.

\subsection*{Why $-1=i^{2}$ is no obstruction}

The cone $P$ satisfies $PP\subseteq P$ and $1\in P$, yet $i$ is incomparable with $0$. Multiplicative closure implies $x^2\in P$ whenever $x\in P$ or $-x\in P$, but imposes no such requirement on elements incomparable with $0$. Thus $i^2=-1\notin P$ is consistent with the lattice order.

\subsection*{Organisation}

Section~\ref{sec:cones} collects the elementary theory of cones, of rational saturation and of
reduced decompositions. Section~\ref{sec:ext} proves the three extension principles
(globalisation, the preservation criterion, increasing unions) and the coordinate and lexicographic extensions used
in the recursion. Section~\ref{sec:val} introduces the valuation, the height filtration and the
well-ordered monoids of heights. Section~\ref{sec:constr} contains the construction and the
theorem on an arbitrary field with such a valuation. Section~\ref{sec:C} specialises to $\C$,
discusses the resulting order and compares it with the cones available in the literature.

\section{Cones and reduced decompositions}\label{sec:cones}

\subsection{Positive cones}\label{sub:cones}

Throughout this paper, $L$ denotes a field of characteristic $0$; in particular $L$ is a vector space over $\Q$. We write $\N_0=\{0,1,2,\ldots\}$ and $\N_{>0}=\{1,2,\ldots\}$. A subset $C\subseteq V$ of a $\Q$-vector space $V$ is a \emph{cone} if
\begin{equation}\label{eq:cone}
0\in C,\qquad C+C\subseteq C,\qquad qC\subseteq C\quad\text{for all }q\in\Q_{\ge0}.
\end{equation}
It is \emph{pointed} if $C\cap(-C)=\{0\}$, and \emph{directed} if $C-C=V$. In a field, a \emph{positive cone} contains $0$, is pointed and is closed under addition and multiplication (Definition~\ref{def:dpo}). Directedness and rational closure are additional properties; both follow when the induced order is a lattice order.

\begin{definition}\label{def:dpo}
A \emph{positive cone} in a field $L$ is a subset $P\subseteq L$ such that
\begin{equation}\label{eq:dpo}
0\in P,\qquad P+P\subseteq P,\qquad P\cdot P\subseteq P,\qquad P\cap(-P)=\{0\}.
\end{equation}
For a positive cone $P$ the relation
\begin{equation}\label{eq:order}
x\le_{P}y\quad:\Longleftrightarrow\quad y-x\in P
\end{equation}
is a partial order which is invariant under translations and compatible with multiplication, in the sense that $a,b\ge_{P}0$ implies $ab\ge_{P}0$. Conversely, every partial order on $L$ with these two properties arises in this way from the positive cone $P=\{x\in L: x\ge_{P}0\}$. A positive cone with $P-P=L$ is \emph{directed}; then $\le_{P}$ is a directed order and $(L,P)$ is a \emph{directed partially ordered field}. A positive cone which is a cone in the sense of \eqref{eq:cone} is \emph{rationally closed}; a directed positive cone need not be rationally closed, and a rationally closed positive cone need not be directed. We identify positive cones with such partial orders, write $x<_{P}y$ when $y-x\in P\setminus\{0\}$, and call $(L,P)$ a \emph{partially ordered field} when $P$ is directed.
\end{definition}

\begin{definition}\label{def:lattice}
Let $C\subseteq V$ be an additive submonoid of a $\Q$-vector space $V$, with $C\cap(-C)=\{0\}$ and $C-C=V$; rational closure is not assumed at this point. If $\le_{C}$ is a lattice order on $V$, that is, if every pair of elements of $V$ has a least upper bound $x\vee y$ and a greatest lower bound $x\wedge y$, then $C$ is a \emph{lattice cone}; by Lemma~\ref{lem:latsat} such a cone is automatically rationally closed. A positive cone $P$ in $L$ which is directed and whose order $\le_{P}$ is a lattice order is called a \emph{lattice positive cone}, and $(L,P)$ is then a \emph{lattice-ordered field}. For $x\in V$ we put
\begin{equation}\label{eq:parts}
x^{+}=0\vee x,\qquad x^{-}=0\vee(-x)
\end{equation}
whenever these suprema exist.
\end{definition}

\begin{remark}\label{rem:nototal}
The field $\C$ admits no total order compatible with its ring operations: in a totally ordered field every square is nonnegative, whereas $i^2=-1<0$. The Birkhoff--Pierce question concerns lattice orders, which need not be total.
\end{remark}

\begin{question}[Birkhoff--Pierce \cite{BP}]\label{q:bp}
Does $\C$ carry a lattice positive cone?
\end{question}

\subsection{Rational saturation}\label{sub:sat}

\begin{lemma}\label{lem:sat}
Let $P$ be a positive cone in $L$ and set
\begin{equation}\label{eq:sat}
\Sat(P)=\{x\in L:\ nx\in P\text{ for some }n\in\N_{>0}\}.
\end{equation}
Then $\Sat(P)$ is a rationally closed positive cone containing $P$, it is directed whenever $P$ is, it is the smallest rationally closed positive cone containing $P$, and $\Sat(\Sat(P))=\Sat(P)$. In particular $P=\Sat(P)$ if and only if $P$ is rationally closed.
\end{lemma}

\begin{proof}
Clearly $P\subseteq\Sat(P)$ and $0\in\Sat(P)$. If $x,y\in\Sat(P)$, choose $m,n\ge1$ with $mx,ny\in P$. Then
\[
mn(x+y)=n(mx)+m(ny)\in P,\qquad mn\,xy=(mx)(ny)\in P.
\]
For $z\in\Sat(P)$, choose $m\in\N_{>0}$ with $mz\in P$. If $q=r/s\in\Q_{\ge0}$ with $r\in\N_0$ and $s\in\N_{>0}$, then
\[
sm\,(qz)=r\,(mz)\in P,
\]
so that $qz\in\Sat(P)$. To prove pointedness, let $x,-x\in\Sat(P)$ and choose $m,n\ge1$ with $mx\in P$ and $-nx\in P$; then $mnx=n(mx)\in P$ and $mn(-x)=m(-nx)\in P$, so that $x=0$. Hence $\Sat(P)$ is a rationally closed positive cone. If $P$ is directed, then $\Sat(P)-\Sat(P)\supseteq P-P=L$. The minimality of $\Sat(P)$ among the rationally closed positive cones containing $P$, and the identity $\Sat(\Sat(P))=\Sat(P)$, are immediate from the definitions.
\end{proof}

\begin{lemma}\label{lem:latsat}
Let $P$ be an additive submonoid of a $\Q$-vector space $V$ with $P\cap(-P)=\{0\}$. If the order $x\le_P y$ defined by $y-x\in P$ is a lattice order on $V$, then $P$ is rationally closed.
\end{lemma}

\begin{proof}
We first prove that $nx\in P$ for $n\in\N_{>0}$ implies $x\in P$. Put $s=\bigvee_{k=0}^{n-1}kx$. Translation invariance gives
\[
s+x=\bigvee_{k=1}^{n}kx.
\]
Every term $kx$ with $1\le k\le n-1$ occurs in the right-hand join, and the remaining term $0$ is dominated by $nx\ge_P0$. Hence $s\le_P s+x$, so $x\in P$. For $p\in P$ and $n\ge1$, apply this implication to $x=p/n$. Closure under addition then gives $(m/n)p\in P$ for every $m\in\N_0$.
\end{proof}

\subsection{Reduced decompositions}\label{sub:reduced}

The classical positive-part criterion and decomposition identities for lattice-ordered abelian groups appear in \cite[Chapter~I, \S~1, Propositions~1.2 and~1.3(vii)]{Johnson1960}. Starting from a pointed cone, without assuming a lattice structure, we formulate the same criterion through a universal property of positive decompositions. Our emphasis is on using this formulation to specify and preserve reduced decompositions as the cone is extended.

The results of this subsection and the additive results of Section~\ref{sec:ext} apply to any $\Q$-vector space in place of $L$; multiplication is needed only where it is explicitly mentioned. Let $C\subseteq L$ be a pointed cone. A \emph{positive representation} of $x\in L$ is a pair $(p,q)\in C\times C$ with $x=p-q$.

\begin{definition}\label{def:reduced}
A positive representation $x=p-q$ is \emph{reduced} if for every positive representation $x=a-b$ there exists $c\in C$ with $a=p+c$ and $b=q+c$.
\end{definition}

A reduced representation has the least first and second coordinates among \emph{all} positive representations of the same element. This condition is stronger than the absence of a nonzero common positive term that can be subtracted from both coordinates.

\begin{example}\label{ex:weak}
The weaker requirement is genuinely insufficient. Let
\[
V=\{(r,s,t,u)\in\Q^{4}:\ r+s=t+u\},\qquad C=V\cap\Q^{4}_{\ge0};
\]
$C$ is a pointed directed cone, since every element of $V$ becomes positive after addition of a
sufficiently large multiple of $(1,1,1,1)$. Put
\[
p=(1,0,1,0),\quad q=(1,0,0,1),\quad a=(0,1,1,0),\quad b=(0,1,0,1);
\]
then $p-q=a-b=(0,0,1,-1)$, and $p$ and $q$ have no nonzero common positive part: if $h\in C$ with
$h\le p,q$, coordinatewise comparison gives $h=(r,0,0,0)$ and $h\in V$ forces $r=0$. Nevertheless
$a-p=(-1,1,0,0)\notin C$, so $(a,b)$ is not of the form $(p+c,q+c)$ with $c\in C$, and $(p,q)$ is
not reduced.
\end{example}

\begin{theorem}\label{thm:reduced}
Let $C\subseteq L$ be a pointed cone, let $x\in L$ and let $x=p-q$ be a positive representation. The following are equivalent:
\begin{enumerate}
\item[\textup{(i)}] $(p,q)$ is reduced;
\item[\textup{(ii)}] $0\vee_{C}x$ exists and equals $p$;
\item[\textup{(iii)}] $0\vee_{C}(-x)$ exists and equals $q$.
\end{enumerate}
If $C$ is a lattice cone, these are further equivalent to
\begin{enumerate}
\item[\textup{(iv)}] $p\wedge_{C}q$ exists and equals $0$.
\end{enumerate}
In particular a reduced representation, when it exists, is unique.
\end{theorem}

\begin{proof}
(ii)$\Rightarrow$(i): let $x=a-b$ be a positive representation. Since $a\ge0$ and $a\ge x$, and since $p=0\vee x$, we have $a\ge p$; thus $c=a-p\in C$ and $b=a-x=q+c$. The implication (iii)$\Rightarrow$(i) is the same statement for $-x$.

(i)$\Rightarrow$(ii): let $z\in C$ with $z\ge x$. Then $x=z-(z-x)$ is a positive representation, so reducedness yields $c\in C$ with $z=p+c$ and $z-x=q+c$; in particular $z\ge p$. Hence $p=0\vee x$. Applied to $-x$ this gives (i)$\Rightarrow$(iii).

Assume now that $C$ is a lattice cone. We first record two lattice identities; both follow from translation invariance and from the fact that $x\mapsto-x$ interchanges $\vee$ and $\wedge$. The first is
\begin{equation}\label{eq:lgroup0}
(w\vee 0)+(w\wedge 0)=w\qquad(w\in L):
\end{equation}
indeed $(w\vee 0)-w=0\vee(-w)=-(w\wedge 0)$. The second is
\begin{equation}\label{eq:lgroup1}
(x\vee y)+(x\wedge y)=x+y\qquad(x,y\in L):
\end{equation}
subtracting $2y$ from both sides reduces it to \eqref{eq:lgroup0} with $w=x-y$. Now let $p,q\in C$ and $u=p\wedge q$. By \eqref{eq:lgroup1} with $x=p$, $y=q$, and by translation invariance,
\begin{equation}\label{eq:lgroup}
(p-q)^{+}=0\vee(p-q)=(p\vee q)-q=p-u,
\end{equation}
so that
\begin{equation}\label{eq:lgroup2}
p=(p-q)^{+}+p\wedge q.
\end{equation}
If $(p,q)$ is reduced, then $p=(p-q)^{+}$ by (ii), so \eqref{eq:lgroup2} gives $p\wedge q=0$, which is (iv). Conversely, if (iv) holds then \eqref{eq:lgroup2} gives $p=(p-q)^{+}=0\vee x$, which is (ii).

Finally, if $(p,q)$ and $(p',q')$ are reduced, then $p=p'$ by (ii), and then $q=q'$.
\end{proof}

\begin{corollary}\label{cor:lattice-crit}
Let $C\subseteq L$ be a pointed cone. Then $C$ induces a lattice order on $L$ if and only if every $x\in L$ has a reduced decomposition. In that case $x^{+}=p$, $x^{-}=q$ for the reduced decomposition $x=p-q$, and
\begin{equation}\label{eq:supinf}
a\vee b=b+(a-b)^{+},\qquad a\wedge b=a-(a-b)^{+}.
\end{equation}
\end{corollary}

\begin{proof}
If $\le_C$ is a lattice order, translation invariance gives $x^+-x=0\vee(-x)=x^-$. Thus $x=x^+-x^-$ is a positive representation, and it is reduced by Theorem~\ref{thm:reduced}. Conversely, if every element has a reduced decomposition, then $x^{+}$ exists for every $x$, and \eqref{eq:supinf} expresses the least upper bound and the greatest lower bound of an arbitrary pair. The identity $a\vee b=b+(a-b)^{+}$ follows from the translation invariance of $\le_C$, and $a\wedge b=a-(a-b)^{+}$ from $a\wedge b=-((-a)\vee(-b))$.
\end{proof}


\section{Extension principles}\label{sec:ext}

\subsection{Order-convex subspaces}\label{sub:convex}

\begin{definition}\label{def:convex}
Let $C\subseteq L$ be a pointed cone and let $U\subseteq L$ be a $\Q$-subspace. Then $U$ is \emph{order convex} with respect to $\le_{C}$ if $0\le_{C}z\le_{C}u$ with $u\in U$ implies $z\in U$. We say that an extension $C\subseteq D$ of pointed cones \emph{preserves reduced decompositions} on $U$ if every pair $(p,q)$ which is the reduced decomposition of some $x\in U$ with respect to $C\cap U$ is also a reduced decomposition of $x$ with respect to $D$. The preservation thus required is preservation in the whole ambient space $L$, not merely in $D\cap U$; preservation in a subspace is a strictly weaker condition and does not suffice for the limit steps below.
\end{definition}

\begin{theorem}[Globalisation]\label{thm:global}
Let $C\subseteq L$ be a pointed directed cone and let $U\subseteq L$ be a $\Q$-subspace such that
\begin{enumerate}
\item[\textup{(a)}] $C\cap U$ induces a lattice order on $U$, and
\item[\textup{(b)}] $(C\setminus U)+U\subseteq C\setminus U$.
\end{enumerate}
Then $U$ is order convex with respect to $\le_{C}$, and for every $x\in U$ the reduced decomposition of $x$ in $U$ is also its reduced decomposition in $L$. In particular, the suprema and infima of pairs of elements of $U$, computed in $U$, agree with those computed in $L$.
\end{theorem}

\begin{proof}
Let $x\in U$ and let $x=p-q$ be its reduced decomposition with $p,q\in C\cap U$. Let $x=a-b$ be any positive representation with $a,b\in C$. If $a\in U$, then $b=a-x\in U$, and reducedness in $U$ gives $c\in C\cap U$ with $a=p+c$, $b=q+c$. If $a\notin U$, then $a\in C\setminus U$ and $-p\in U$, so that by (b)
\[
c:=a-p\in(C\setminus U)+U\subseteq C\setminus U,
\]
and $b-q=a-p=c$. In both cases $a=p+c$ and $b=q+c$ with $c\in C$. Hence $(p,q)$ is reduced in $L$.

For order convexity, let $0\le_{C}z\le_{C}u$ with $u\in U$ and suppose that $z\notin U$. Then $z\in C\setminus U$ and $u-z\in C$, so $-u\in U$ gives
\[
z-u\in(C\setminus U)+U\subseteq C\setminus U.
\]
Thus $z-u\in C$, that is $z\ge u$; together with $z\le u$ and the pointedness of $C$ this forces $z=u\in U$, a contradiction. The assertion about suprema and infima follows from the preserved positive part of the difference and the identities \eqref{eq:supinf}.
\end{proof}
\begin{corollary}[Criterion for preservation]\label{cor:preserve}
Let $C\subseteq D$ be pointed directed cones in $L$ and let $U\subseteq V$ be $\Q$-subspaces. Assume that the pairs $(C,U)$ and $(D,V)$ satisfy the hypotheses of Theorem~\ref{thm:global}, that $D\cap U=C\cap U$, and that $U$ is order convex in $V$ with respect to $\le_{D\cap V}$. Then $C\subseteq D$ preserves reduced decompositions on $U$.
\end{corollary}

\begin{proof}
Let $x\in U$ and let $p=0\vee_{C\cap U}x$; by Theorem~\ref{thm:global} applied to $(C,U)$, $p$ is also $0\vee_{C}x$. Let $r=0\vee_{D\cap V}x$, which exists by hypothesis and which is a common upper bound of $0$ and $x$ in $V$. Since $p$ is such a bound as well, $0\le_{D}r\le_{D}p$. As $p\in D$ and $U$ is order convex in $V$ with respect to $\le_{D\cap V}$, we get $r\in U$; and $r\in D\cap U=C\cap U$ because $0\le_{D}r$ and $r\le_{D}p\in D$. Hence $r$ is a common upper bound of $0$ and $x$ in $U$ with respect to $\le_{C\cap U}$, so $p\le_{C}r$, that is $r-p\in C\cap U\subseteq D$. Together with $p-r\in D$ and the pointedness of $D$ this gives $r=p$.

Finally, Theorem~\ref{thm:global} applied to the pair $(D,V)$ shows that $V$ is order convex in $L$ with respect to $\le_{D}$ and that reduced decompositions with respect to $D\cap V$ are reduced with respect to $D$; this completes the comparison, for every $x\in U$, between the decomposition of $x$ in $U$ and its decomposition in $V$.
\end{proof}

\begin{theorem}[Increasing unions]\label{thm:union}
Let $\lambda\ge1$ be an ordinal, and let $(C_\beta)_{\beta<\lambda}$ and $(U_\beta)_{\beta<\lambda}$ be increasing families of pointed directed cones in $L$ and of $\Q$-subspaces of $L$, respectively. Put $A_\beta=C_\beta\cap U_\beta$. Assume that:
\begin{enumerate}
\item[\textup{(a)}] $A_\beta$ is a lattice cone on $U_\beta$, and its reduced decompositions are reduced with respect to $C_\beta$;
\item[\textup{(b)}] for $\beta<\gamma<\lambda$, every reduced decomposition with respect to $A_\beta$ is reduced with respect to $C_\gamma$, and $U_\beta$ is order convex in $U_\gamma$ with respect to $A_\gamma$.
\end{enumerate}
Set $C=\bigcup_{\beta<\lambda}C_\beta$ and $U=\bigcup_{\beta<\lambda}U_\beta$. Then $C$ is a pointed directed cone, $C\cap U=\bigcup_{\beta<\lambda}A_\beta$ is a lattice cone on $U$, and every reduced decomposition with respect to $A_\beta$ is reduced with respect to $C$. Moreover
\[
C\cap U_\beta=A_\beta,
\]
and $U_\beta$ is order convex in $U$ with respect to $C\cap U$. If $L$ is a field and every $C_\beta$ is closed under multiplication, then so is $C$.
\end{theorem}

\begin{proof}
Closure under addition and nonnegative rational scaling follows by taking a common stage for finitely many elements. The same argument proves multiplicative closure when assumed. If $x,-x\in C$, both lie in a common $C_\gamma$, so $x=0$. Directedness follows from $C_0-C_0=L$.

Let $x\in U_\beta$ and let $x=p-q$ be its reduced decomposition with respect to $A_\beta$. For any representation $x=a-b$ with $a,b\in C$, choose $\gamma\ge\beta$ such that $a,b\in C_\gamma$. Assumption (a) applies if $\gamma=\beta$, and (b) applies if $\gamma>\beta$. In either case there is $c\in C_\gamma$ with $a=p+c$ and $b=q+c$. Thus $(p,q)$ is reduced with respect to $C$. If $a,b\in C\cap U$, then $c=a-p$ also lies in $U$, proving reducedness with respect to $C\cap U$. Every element of $U$ is treated in this way, so Corollary~\ref{cor:lattice-crit} proves the lattice assertion. The equality $C\cap U=\bigcup_\beta A_\beta$ follows by choosing a common stage for membership in $C$ and in $U$.

For the restriction assertion, take $x\in C\cap U_\beta$ and its old reduced decomposition $x=p-q$. Comparing it with $x=x-0$ in $C$ gives $0=q+c$ for some $c\in C$. Since $q\in C$ and $-q=c\in C$, pointedness gives $q=0$, and hence $x=p\in A_\beta$.

Finally, if $0\le_{C\cap U}z\le_{C\cap U}u$ with $u\in U_\beta$, choose $\gamma\ge\beta$ such that $z,u\in U_\gamma$ and $z,u-z\in C_\gamma$. Then $0\le_{A_\gamma}z\le_{A_\gamma}u$, so $z\in U_\beta$ by (b), or immediately if $\gamma=\beta$. This proves order convexity.
\end{proof}

\subsection{Coordinate and lexicographic extensions}\label{sub:linext}

\begin{lemma}\label{lem:linext}
Let $A$ be a lattice cone on a $\Q$-subspace $U\subseteq L$.
\begin{enumerate}
\item[\textup{(a)}] \emph{(Coordinate extension.)} Let $V=U\oplus\bigoplus_{j\in J}\Q e_j$ and $B=A+\bigoplus_{j\in J}\Q_{\ge0}e_j$. Then $B$ is a lattice cone on $V$, $B\cap U=A$, $U$ is order convex in $V$, and reduced decompositions with respect to $A$ remain reduced with respect to $B$. Moreover
\[
\Big(u+\sum_jr_je_j\Big)^+=u^++\sum_j\max\{r_j,0\}e_j.
\]
\item[\textup{(b)}] \emph{(Lexicographic extension.)} Let $V=\Q e\oplus U$ and
\[
B=A\cup\{re+u:r\in\Q_{>0},\ u\in U\}.
\]
Then $B$ is a lattice cone on $V$, $B\cap U=A$, $U$ is order convex in $V$, and reduced decompositions with respect to $A$ remain reduced with respect to $B$. Its positive parts are
\[
(re+u)^+=
\begin{cases}
re+u,&r>0,\\
u^+,&r=0,\\
0,&r<0.
\end{cases}
\]
\end{enumerate}
\end{lemma}

\begin{proof}
(a) Addition, nonnegative rational scaling and directedness follow from the direct sum. Pointedness and $B\cap U=A$ follow by comparing coordinates. For $v=u+\sum_jr_je_j$, put
\[
p=u^++\sum_j\max\{r_j,0\}e_j,\qquad
q=u^-+\sum_j\max\{-r_j,0\}e_j.
\]
Then $v=p-q$. If also $v=a-b$ with $a=a_0+\sum_js_je_j$ and $b=b_0+\sum_jt_je_j$ in $B$, then $u=a_0-b_0$ and $r_j=s_j-t_j$. Reducedness in $A$ gives $a_0-u^+=b_0-u^-\in A$, while
\[
s_j-\max\{r_j,0\}=t_j-\max\{-r_j,0\}\ge0.
\]
Consequently $a-p=b-q\in B$, so $(p,q)$ is reduced. Corollary~\ref{cor:lattice-crit} now gives the lattice property and the formula for $v^+$. For $v\in U$ this is the old decomposition, proving preservation. If $0\le_B z\le_B u$ with $u\in U$, the new coordinates of both $z$ and $u-z$ are nonnegative and are negatives of each other. They are all zero, so $z\in U$.

(b) The sign of the $e$-coordinate shows that $B$ is pointed and that $B\cap U=A$; the cone axioms are immediate. An element $re+u$ is a difference of two elements with positive $e$-coordinate, by choosing a rational number larger than $0$ and $r$, so $B$ is directed. If $r>0$, the element $v=re+u$ is positive and hence $v^+=v$. If $r<0$, then $-v$ is positive and hence $v^+=0$. If $r=0$, $u^+$ is a common upper bound of $0,u$. Any other common upper bound $z\in B$ either lies in $A$, where $z\ge_Au^+$, or has positive $e$-coordinate, in which case $z-u^+\in B$. Thus $v^+=u^+$. These formulae, applied also to $-v$, give $v=v^+-v^-$, and Theorem~\ref{thm:reduced} and Corollary~\ref{cor:lattice-crit} prove the lattice property and preservation on $U$. Finally, if $0\le_B z\le_B u$ with $u\in U$, the $e$-coordinates of $z$ and $u-z$ are nonnegative and opposite. They vanish, so $z\in U$.
\end{proof}

\section{A \texorpdfstring{$\Q$}{Q}-valued valuation and heights}\label{sec:val}

\subsection{The hypotheses}\label{sub:hyp}

By a \emph{valuation} on $L$ with values in $\Q$ we mean a map $v\colon L^{\times}\to\Q$ with
\[
v(xy)=v(x)+v(y),\qquad v(x+y)\ge\min\{v(x),v(y)\}\quad(x+y\ne0),
\]
extended by $v(0)=+\infty$. We shall work under the following assumptions.

\begin{center}
\fbox{\parbox{0.94\textwidth}{
\textbf{Hypotheses.}
\begin{enumerate}
\item[\textup{(H1)}] $v\colon L^{\times}\to\Q$ is a valuation with $v(\Q^{\times})=\{0\}$ and $v(L^{\times})=\Q$;
\item[\textup{(H2)}] $t\in L^{\times}$ satisfies $v(t)=-1$;
\item[\textup{(H3)}] there is a family $(t^{r})_{r\in\Q}$ in $L^{\times}$ with $t^{0}=1$, $t^{1}=t$, $t^{r}t^{s}=t^{r+s}$ and $v(t^{r})=-r$ for all $r,s\in\Q$.
\end{enumerate}}}
\end{center}

Hypothesis (H1) makes the height filtration $\Q$-indexed, while (H3) supplies multiplicatively compatible representatives of its heights. We shall also use the following consequence of $v(\Q^\times)=\{0\}$: every nonzero element of $L$ algebraic over $\Q$ has valuation zero. Indeed, in a polynomial relation with nonzero constant and leading coefficients, a positive valuation would make the constant term uniquely least valued, and a negative valuation would make the leading term uniquely least valued. Both contradict the valuation inequality applied to a sum equal to zero. In particular $t$ is transcendental over $\Q$, and $v(i)=0$ whenever $i^2=-1$. The next lemma produces (H1)--(H3) inside $\C$.

\begin{lemma}\label{lem:exist}
Let $T$ be a transcendence basis of $\C$ over $\Q$, let $t\in T$ and put $F_0=\Q(T\setminus\{t\})$. Then $\C$ satisfies \textup{(H1)--(H3)}, with $v$ trivial on $F_0^{\times}$.
\end{lemma}

\begin{proof}
Every nonzero element of $F_0(t)=\Q(T)$ is a quotient $f(t)/g(t)$ of nonzero polynomials $f,g\in F_0[X]$. Put
\[
v_0(f/g)=\deg g-\deg f.
\]
This is a valuation on $F_0(t)$: it is additive on products, and $v_0(f+g)\ge\min\{v_0(f),v_0(g)\}$ for polynomials, with equality when the degrees differ; the same then holds for quotients. Clearly $v_0(F_0^{\times})=\{0\}$, $v_0(t)=-1$ and $v_0(F_0(t)^{\times})=\Z$.

Since $\C$ is algebraic over $F_0(t)$, the valuation $v_0$ extends to a valuation $w$ on $\C$, by the extension theorem for valuations on algebraic field extensions; see \cite[Section 3.1]{EP}. Let $\Gamma=w(\C^{\times})$. For $x\in\C^{\times}$ choose a nontrivial relation $\sum_{j=0}^{n}a_jx^{j}=0$ with $a_j\in F_0(t)$ and $n\ge1$ minimal. Among the numbers $w(a_j)+jw(x)$ with $a_j\ne0$, the minimum is attained at least twice, say for $j\ne k$; then
\[
(j-k)w(x)=w(a_k)-w(a_j)\in\Z.
\]
For each $\gamma\in\Gamma$, this gives $n\gamma=m$ for some $n\in\N_{>0}$ and $m\in\Z$, where $\Z$ is identified with the value group of $F_0(t)$. Define $\phi(\gamma)=m/n$. If also $n'\gamma=m'$, then $n'm=nm'$, so $\phi$ is well defined. If $n_1\gamma_1=m_1$ and $n_2\gamma_2=m_2$, then
\[
n_1n_2(\gamma_1+\gamma_2)=n_2m_1+n_1m_2,
\]
and hence $\phi(\gamma_1+\gamma_2)=\phi(\gamma_1)+\phi(\gamma_2)$. Moreover, since $n>0$ and $\Gamma$ is an ordered abelian group,
\[
\gamma>0\quad\Longleftrightarrow\quad n\gamma=m>0
\quad\Longleftrightarrow\quad m/n>0.
\]
Thus $\phi$ is an order-preserving injection and fixes $\Z$. We hence identify $\Gamma$ with a subgroup of $\Q$. Since $\C$ is algebraically closed, $\Gamma$ is divisible: an $n$th root of an element of value $\gamma$ has value $\gamma/n$. A divisible subgroup of $\Q$ containing $\Z$ must be all of $\Q$, so $\Gamma=\Q$ under this identification.

It remains to choose the compatible representatives required in (H3). Choose $z_0=t$ and, recursively for $n\ge1$, $z_n\in\C^\times$ with $z_n^n=z_{n-1}$; this is possible since $\C$ is algebraically closed. For $r=m/n!$ with $m\in\Z$ put $t^{r}=z_{n}^{m}$. If $r=m/n!=m'/n'!$ and $N=\max\{n,n'\}$, then $z_{N}^{N!/n!}=z_{n}$ by induction, so that
\[
z_{n}^{m}=\big(z_{N}^{N!/n!}\big)^{m}=z_{N}^{\,mN!/n!}=z_{N}^{\,m'N!/n'!}=z_{n'}^{m'},
\]
and $t^{r}$ is well defined; the identities $t^{r}t^{s}=t^{r+s}$, $t^{0}=1$, $t^{1}=t$ follow from the same computation. Finally $n!\,w(z_{n})=w(t)=-1$, whence $w(t^{r})=m\,w(z_{n})=-m/n!=-r$. Thus $\Q=\Gamma$, and $v:=w$ satisfies \textup{(H1)--(H3)}.
\end{proof}

\subsection{Heights and layers}\label{sub:height}

Assume (H1)--(H3). The \emph{height} of $x\in L$ is
\begin{equation}\label{eq:height}
d(x)=-v(x)\quad(x\ne0),\qquad d(0)=-\infty,
\end{equation}
a value in $\Q\cup\{-\infty\}$ by (H1). For a real number $a\in\R$ we put
\begin{equation}\label{eq:levels}
L_{<a}=\{x\in L:\ d(x)<a\},\qquad L_{\le a}=\{x\in L:\ d(x)\le a\},
\end{equation}
and for $s\in\Q$,
\begin{equation}\label{eq:layers}
H_{s}=\{ct^{s}+u:\ c\in\Q_{>0},\ u\in L_{<s}\}.
\end{equation}
The selected layers $H_s$ consist of elements whose class in $L_{\le s}/L_{<s}$ is a positive rational multiple of the class of $t^s$. They need not exhaust the elements of height $s$, and no order on the whole residue field is assumed. The subspaces $L_{<a}$ and $L_{\le a}$ are defined for real cut values $a$, which are needed at the limit steps of Section~\ref{sec:constr}. These limit cuts need not be rational, although every nonzero element still has rational height.

\begin{lemma}\label{lem:layers}
For all $a\in\R$, $s,r\in\Q$ and $\lambda\in\Q_{>0}$ the following hold.
\begin{enumerate}
\item[\textup{(a)}] $L_{<a}$ is a $\Q$-subspace of $L$, and $L=\bigcup_{a\in\Q}L_{<a}$.
\item[\textup{(b)}] $d(x+y)\le\max\{d(x),d(y)\}$, with equality if $d(x)\ne d(y)$; and $d(xy)=d(x)+d(y)$, with the convention $-\infty+r=-\infty$ for $r\in\Q\cup\{-\infty\}$.
\item[\textup{(c)}] Every $x\in H_s$ satisfies $d(x)=s$, and its coefficient $c$ in \eqref{eq:layers} is unique. Hence the layers $H_s$ are pairwise disjoint and $H_s\cap(-H_r)=\emptyset$ for all $s,r$.
\item[\textup{(d)}] $\lambda H_s=H_s$; $H_s+H_r\subseteq H_{\max\{s,r\}}$; $H_sH_r=H_{s+r}$; and $H_s+L_{<s}=H_s$.
\item[\textup{(e)}] $H_s\subseteq L_{\le s}$, and $L_{<s}\subseteq L_{<r}$ for $s<r$.
\end{enumerate}
\end{lemma}

\begin{proof}
Parts (a), (b) and (e) restate the axioms of a valuation, with $d(x)$ finite for $x\ne0$ by (H1), which gives the last assertion of (a). For (c), let $ct^{s}+u=c't^{s}+u'$ with $c,c'\in\Q_{>0}$, $u,u'\in L_{<s}$ and $c\ne c'$. Then $(c-c')t^{s}=u'-u$, and by (b) and (H3) the left-hand side has height $s$, whereas the right-hand side has height $<s$; this is impossible. Hence $c=c'$ and $u=u'$. In particular $H_s\cap(-H_s)=\emptyset$, since an element of $H_s$ has leading coefficient $c>0$ while its negative has leading coefficient $-c<0$. For (d), the first assertion is clear. For the inclusion $H_sH_r\subseteq H_{s+r}$,
\[
(ct^{s}+u)(c't^{r}+u')=cc't^{s+r}+\big(cu't^{s}+c't^{r}u+uu'\big),
\]
where the bracket has height $<s+r$ by (b). For the reverse inclusion, let $ct^{s+r}+u\in H_{s+r}$ with $d(u)<s+r$; then
\[
ct^{s+r}+u=t^{s}\big(ct^{r}+t^{-s}u\big),
\]
and $d(t^{-s}u)=-s+d(u)<r$, so that $t^{s}\in H_s$ and $ct^{r}+t^{-s}u\in H_r$; hence $H_{s+r}\subseteq H_sH_r$. The remaining assertions are immediate from (b).
\end{proof}

\subsection{Well-ordered monoids of heights}\label{sub:monoid}

A subset $D\subseteq\Q$ generates the additive monoid
\begin{equation}\label{eq:S}
S=S(D)=\Big\{\sum_{d\in D}n_dd:\ n_d\in\N_0,\ \text{almost all }n_d=0\Big\}.
\end{equation}
Recall that a totally ordered set is \emph{well ordered} if every nonempty subset has a least element.

\begin{lemma}\label{lem:monoid}
Let $D\subseteq\Q\cap[1,\infty)$ be well ordered and contain $1$, and put $S=S(D)$.
\begin{enumerate}
\item[\textup{(a)}] $S$ is well ordered and unbounded above.
\item[\textup{(b)}] For a finite real number $a\ge1$, let $b=\min\{s\in S:s>a\}$. If $h\in\Q$ satisfies $a<h<b$, then the monoid generated by $D\cup\{h\}$ equals $S+\N_0h$, is well ordered, and has $h$ as its least element strictly greater than $a$.
\end{enumerate}
\end{lemma}

\begin{proof}
Every sequence in a well-ordered set has an infinite nondecreasing subsequence: successively choose an index at which the minimum of the remaining tail is attained. It follows that a finite sum of well-ordered subsets of $\Q$ is well ordered. Otherwise, a strictly decreasing sequence of such sums would, by successively passing to subsequences for each of its finitely many summands, have all summands nondecreasing, a contradiction. A finite union of well-ordered subsets of a linearly ordered set is also well ordered, by comparing the minima of its nonempty intersections.

(a) For $N\ge0$, every $s\in S$ with $s\le N$ is a sum of at most $\lfloor N\rfloor$ elements of $D$, because each summand is at least $1$. Thus $S\cap[0,N]$ is contained in a finite union of well-ordered finite sums of $D$, including the empty sum $\{0\}$. It is well ordered. For any nonempty subset $E\subseteq S$, choose $s\in E$; a least element of $E\cap[0,s]$ is a least element of $E$. Finally, $\N_0\subseteq S$ proves unboundedness.

(b) The generated monoid is $S'=S+\N_0h$. For each $N\ge0$, its part in $[0,N]$ is contained in the finite union of $(S+nh)\cap[0,N]$ for $0\le n\le N/h$, so the preceding argument proves well-ordering. An element with $n=0$ and value $>a$ is at least $b>h$; one with $n\ge1$ is at least $h$, and $h$ itself belongs to $S'$. This proves the last assertion.
\end{proof}


\section{Construction of a lattice order}\label{sec:constr}

\subsection{Seeds}\label{sub:seed}

Assume (H1)--(H3). A \emph{seed} is a pair $(V_0,A_0)$ consisting of a finite-dimensional $\Q$-subspace $V_0\subseteq L_{\le1}$ and a lattice cone $A_0\subseteq V_0$ such that
\begin{equation}\label{eq:seedcone}
P_0:=A_0\cup\bigcup_{n\ge2}H_n
\end{equation}
is a positive cone in $L$. Two seeds will be used.

\begin{lemma}\label{lem:seeds}
The following pairs are seeds.
\begin{enumerate}
\item[\textup{(i)}] $(\Q1,\Q_{\ge0})$ is a seed.
\item[\textup{(ii)}] If $i\in L$ satisfies $i^2=-1$, put
\[
p=\tfrac12(t+i),\qquad q=\tfrac12(t-i),\qquad
V_0=\Q1\oplus\Q p\oplus\Q q,\qquad
A_0=\Q_{\ge0}1+\Q_{\ge0}p+\Q_{\ge0}q.
\]
Then $(V_0,A_0)$ is a seed, $1,p,q$ are linearly independent, and $(p,q)$ is the reduced decomposition of $i$ with respect to $A_0$.
\end{enumerate}
In both cases $P_0$ is rationally closed and directed,
$P_0\setminus\{0\}\subseteq\Q_{>0}\cup\bigcup_{n\ge1}H_n$,
$P_0\cap V_0=A_0$, and every reduced decomposition with respect to $A_0$ is reduced with respect to $P_0$.
\end{lemma}

\begin{proof}
For seed (ii), the element $t$ is transcendental over $\Q(i)$, since it has nonzero valuation and $\Q(i)/\Q$ is algebraic. A relation
$a+bp+cq=a+(b+c)t/2+(b-c)i/2=0$ with $a,b,c\in\Q$ first gives $b+c=0$, then $a=0$ and $b-c=0$, because $i\notin\Q$. Thus $1,p,q$ are independent. The cones $A_0$ in both cases are coordinate lattice cones. We have $d(1)=0$ and $d(t)=1$; in case (ii), $d(i)=0$ and $d(p)=d(q)=1$. Thus, in both cases, every nonzero element of $A_0$ is either a positive rational or belongs to $H_1$.

The union \eqref{eq:seedcone} is closed under addition and nonnegative rational scaling: $A_0$ is a cone, the layers have the properties in Lemma~\ref{lem:layers}, and $A_0+H_n\subseteq H_n$ for $n\ge2$. For multiplication, zero factors give zero. Two elements of $A_0$ have product in $A_0$ if one is rational; otherwise both belong to $H_1$ and their product belongs to $H_2$. A nonzero element of $A_0$ times an element of $H_n$ lies in $H_n$ or $H_{n+1}$, according as the first factor is rational or has height $1$. Finally $H_mH_n=H_{m+n}$ for $m,n\ge2$. Hence $P_0$ is multiplicatively closed. Its support in $\Q_{>0}$ and positive rational-leading layers of positive height proves pointedness. Given $x\in L$, choose an integer $m>\max\{1,d(x)\}$. Then $t^m+x,t^m\in H_m\subseteq P_0$, which proves directedness.

The whole layer $H_1$ need not belong to $P_0$. For seed (i), $t\notin P_0$. For seed (ii), $t=p+q\in A_0$, but $t+2i=3p-q\notin A_0$ and its height is $1$, so $t+2i\notin P_0$.

Since $V_0\subseteq L_{\le1}$, the layers with integer height at least $2$ are disjoint from $V_0$ and stable under addition of $V_0$. Thus $P_0\cap V_0=A_0$, and Theorem~\ref{thm:global} proves global reducedness and order convexity of $V_0$. In case (ii), $i=p-q$ and the two coordinates have disjoint supports in $A_0$, so Theorem~\ref{thm:reduced} gives the asserted reduced decomposition.
\end{proof}

\subsection{Configurations and the extension steps}\label{sub:step}

A \emph{configuration} is a tuple $(D,a,S,U,A,P)$ satisfying:
\begin{enumerate}
\item[\textup{(I1)}] $a\in\R_{\ge1}$ is finite, $D\subseteq\Q\cap[1,a]$ is well ordered, $1\in D$, $\sup D=a$, and $S=S(D)$.
\item[\textup{(I2)}] $U\subseteq L_{\le a}$ is a $\Q$-subspace and $A$ is a lattice cone on $U$.
\item[\textup{(I3)}] $P=A\cup\bigcup_{s\in S,\ s>a}H_s$ is a rationally closed directed positive cone on $L$, and
\[
P\setminus\{0\}\subseteq\Q_{>0}\cup\bigcup_{s\in S\setminus\{0\}}H_s.
\]
\item[\textup{(I4)}] $A=P\cap U$, $U$ is order convex in $L$ with respect to $P$, and every reduced decomposition with respect to $A$ is reduced with respect to $P$.
\end{enumerate}
We call a configuration \emph{full} if it also satisfies
\[
\text{(I5)}\qquad L_{<a}\subseteq U.
\]
Each of the two seeds in Lemma~\ref{lem:seeds} gives a configuration with $a=1$, $D=\{1\}$, $S=\N_0$, $U=V_0$ and $A=A_0$. These initial configurations are not required to satisfy (I5); all later configurations constructed below do. We use only these two seeds, whose cones $A_0$ both contain $1$. The general definition of a seed does not require $1\in A_0$: for instance $(\Q t,\Q_{\ge0}t)$ is also a seed.

\begin{proposition}[Successor step]\label{prop:succ}
Let $(D,a,S,U,A,P)$ be a configuration and let $b=\min\{s\in S:s>a\}$. Choose $h\in\Q$ with $a<h<b$, put
\[
D'=D\cup\{h\},\qquad S'=S+\N_0h,\qquad U'=\Q t^h\oplus L_{<h},
\]
and choose representatives $x_j\in L_{<h}$ of a basis $(x_j+U)_{j\in J}$ of $L_{<h}/U$. Then
\begin{equation}\label{eq:decomp}
U'=U\oplus\Q t^h\oplus\bigoplus_{j\in J}\Q(t^h+x_j).
\end{equation}
Set
\begin{equation}\label{eq:Aprime}
A'=A+\Q_{\ge0}t^h+\sum_{j\in J}\Q_{\ge0}(t^h+x_j),\qquad
P'=A'\cup\bigcup_{\substack{s\in S'\\s>h}}H_s.
\end{equation}
All sums are finite. Then $(D',h,S',U',A',P')$ is a full configuration, $P\subseteq P'$, $A'\setminus A\subseteq H_h$, and
\[
A'\cap U=P'\cap U=A.
\]
The subspace $U$ is order convex in $U'$ with respect to $A'$, and every reduced decomposition with respect to $A$ is reduced with respect to both $A'$ and $P'$.
\end{proposition}

\begin{proof}
We have $U\subseteq L_{\le a}\subseteq L_{<h}$. In a relation among the summands in \eqref{eq:decomp}, the coefficient of $t^h$ must vanish because all remaining terms have height $<h$; comparison in $L_{<h}/U$ then makes each coefficient of $x_j$ vanish. This proves directness. The same basis property proves that the summands span $U'$. Lemma~\ref{lem:linext}(a) makes $A'$ a lattice cone, proves $A'\cap U=A$, and proves order convexity and local preservation. The set $D'$ is well ordered, $\sup D'=h$, and Lemma~\ref{lem:monoid}(b) gives the required monoid $S'$, proving (I1) and (I2).

If a positive coordinate expression for an element of $A'$ uses at least one new coordinate, the sum of its new coefficients is positive. That expression has the form $ct^h+u$ with $c\in\Q_{>0}$ and $u\in L_{<h}$. Hence $A'\setminus A\subseteq H_h$, and every nonzero element of $A'$ is either a positive rational or belongs to some $H_s$ with $s\in S'\setminus\{0\}$.

Every old layer indexed by $s\in S$ with $s>a$ has $s\ge b>h$, so it is still present in $P'$. Thus $P\subseteq P'$. The set $P'$ is closed under addition and nonnegative rational scaling: these properties hold within $A'$, addition of $A'\subseteq L_{\le h}$ leaves each higher layer unchanged, and sums of higher layers are controlled by Lemma~\ref{lem:layers}(d). The support just proved gives pointedness.

For multiplicative closure, it suffices to consider nonzero factors $x,y\in P'$:
\begin{enumerate}
\item If $x,y\in A$, then $xy\in P\subseteq P'$.
\item If $x\in A'\setminus A$ and $y\in A$, then $x\in H_h$. If $y\in\Q_{>0}$, the rational closure of $A'$, established in Lemma~\ref{lem:linext}(a), gives $xy\in A'$. Otherwise $y\in H_r$ for some $r\in S\setminus\{0\}$, so $r\ge1$ and $xy\in H_{h+r}\subseteq P'$.
\item If $x,y\in A'\setminus A$, then $xy\in H_{2h}\subseteq P'$.
\item If $x\in H_s$ with $s\in S'$ and $s>h$, then a positive rational factor $y$ keeps the product in $H_s$. Every other possible nonzero factor belongs to some $H_r$ with $r\in S'\setminus\{0\}$, and the product belongs to $H_{s+r}\subseteq P'$.
\end{enumerate}
These cases, with the two factors interchanged when necessary, exhaust all products. For directedness, choose $m\in S'$ with $m>\max\{h,d(x)\}$ for any $x\in L$; then $x=(t^m+x)-t^m$ is a difference of two elements of $P'$. This proves (I3).

All higher layers have height $>h$, so $P'\cap U'=A'$. They are stable under addition of $U'\subseteq L_{\le h}$. Theorem~\ref{thm:global}, applied to the \emph{new} pair $(P',U')$, proves (I4). Since $U\subseteq U'$, the restriction equality $P'\cap U=A$ follows. Together with the order convexity already proved, Corollary~\ref{cor:preserve}, applied to $(P,U)$ and $(P',U')$, gives the asserted preservation with respect to $P'$. Finally $L_{<h}\subseteq U'$ proves (I5).
\end{proof}

\begin{proposition}[Limit step]\label{prop:limit}
Let $\lambda$ be a nonzero limit ordinal. Consider configurations
\[
(D_\beta,a_\beta,S_\beta,U_\beta,A_\beta,P_\beta)_{\beta<\lambda}
\]
with strictly increasing cuts $a_\beta$ and finite supremum $a=\sup_{\beta<\lambda}a_\beta$. Assume that for $\beta<\gamma<\lambda$:
\begin{enumerate}
\item[\textup{(L1)}] $D_\beta\subseteq D_\gamma$, $U_\beta\subseteq U_\gamma$, $A_\beta\subseteq A_\gamma$ and $P_\beta\subseteq P_\gamma$, and
$D_\gamma\setminus D_\beta\subseteq(a_\beta,\infty)$;
\item[\textup{(L2)}] $P_\gamma\cap U_\beta=A_\beta$, every reduced decomposition with respect to $A_\beta$ is reduced with respect to $P_\gamma$, and $U_\beta$ is order convex in $U_\gamma$ with respect to $A_\gamma$.
\end{enumerate}
Put $D=\bigcup_{\beta<\lambda}D_\beta$, $S=S(D)$ and $V=L_{<a}$, and assume $\bigcup_{\beta<\lambda}U_\beta=V$. Write $A_V=\bigcup_{\beta<\lambda}A_\beta$.
\begin{enumerate}
\item[\textup{(a)}] If $a\notin S$, set $U=V$ and $A=A_V$.
\item[\textup{(b)}] If $a\in S$, set $U=\Q t^a\oplus V$ and $A=A_V\cup H_a$.
\end{enumerate}
Then in either case
\[
P=A\cup\bigcup_{\substack{s\in S\\s>a}}H_s
=\bigcup_{\beta<\lambda}P_\beta
\]
makes $(D,a,S,U,A,P)$ a full configuration. Conditions \textup{(L1)} and \textup{(L2)} continue to hold with the new stage in place of $\gamma$, and reduced decompositions with respect to $A$ are reduced with respect to $P$.
\end{proposition}

\begin{proof}
Since $\lambda$ is a limit ordinal and the cuts strictly increase, $a_\beta<a$ for every $\beta<\lambda$. To prove that $D$ is well ordered, take a nonempty $E\subseteq D$ and choose $e\in E\cap D_\beta$. Every element of $D\setminus D_\beta$ is greater than $a_\beta\ge e$, by (L1). Therefore the least element of $E\cap D_\beta$ is a least element of $E$. Also $1\in D\subseteq\Q\cap[1,a]$ and $\sup D=a$. Every element of $S(D)$ uses only finitely many generators, all lying in a common $D_\beta$, so
\[
S=\bigcup_{\beta<\lambda}S_\beta.
\]
Lemma~\ref{lem:monoid} proves that $S$ is well ordered.

Set $C=\bigcup_{\beta<\lambda}P_\beta$. Apply Theorem~\ref{thm:union} using (I4) at each stage and (L2) between stages. It follows that $C$ is a rationally closed directed positive cone, that
$C\cap V=A_V$ is a lattice cone on $V$, and that all earlier reduced decompositions are reduced with respect to $C$. Moreover $C\cap U_\beta=A_\beta$, and each $U_\beta$ is order convex in $V$ with respect to $A_V$.

We claim
\[
C=A_V\cup\bigcup_{\substack{s\in S\\s\ge a}}H_s.
\]
Indeed, if $x\in C\setminus V$, then $d(x)\ge a$. By (I3) at a stage containing $x$, it belongs to some $H_s$ with $s\in S$ and $s\ge a$. Conversely, if $s\in S$ and $s\ge a$, choose $\beta$ with $s\in S_\beta$. Since $s\ge a>a_\beta$, the whole layer $H_s$ is already contained in $P_\beta$. This proves the claim.

If $a\notin S$, the claim gives $C=A_V\cup\bigcup_{s\in S,\ s>a}H_s$, so the specified $P$ equals $C$. If $a\in S$, then $a\in\Q$ and
$H_a=\{ct^a+v:c\in\Q_{>0},v\in V\}$. Lemma~\ref{lem:linext}(b) makes $A=A_V\cup H_a$ the lexicographic lattice cone on $U=\Q t^a\oplus V$. The same claim again gives $P=C$. Thus, in both cases, pointedness, directedness, rational closure and multiplicative closure are inherited from the increasing union; no additional product cases are required. The support assertion in (I3) also follows from the support of its stages.

We have $U\subseteq L_{\le a}$ and $P\cap U=A$, since every remaining layer has height $>a$. These remaining layers are stable under addition of $U$. Theorem~\ref{thm:global}, now applied to $(P,U)$, proves (I4). Also $L_{<a}=V\subseteq U$, proving (I5).

Finally, $P=C$ already preserves all earlier decompositions and satisfies $P\cap U_\beta=A_\beta$. The subspace $U_\beta$ is order convex in $V$, and $V$ is order convex in $U$ by the lexicographic lemma in case (b), or by equality in case (a). The induced orders agree on these subspaces, so order convexity is transitive here. This proves (L2) for the new stage. The inclusions in (L1) are immediate, and its assertion about new elements of $D$ follows from the corresponding assertion at the earlier stages.
\end{proof}

\subsection{The main theorem}\label{sub:main}

\begin{theorem}\label{thm:main}
Assume \textup{(H1)--(H3)}. Then $L$ carries a lattice positive cone $P$ with
\begin{equation}\label{eq:props}
1\in P,\qquad \Q_{>0}\subseteq P,\qquad P\cap\Q=\Q_{\ge0},\qquad
\{x\in P\setminus\{0\}:x^{-1}\in P\}=\Q_{>0}.
\end{equation}
If $L$ also contains $i$ with $i^2=-1$, then $P$ may be chosen so that
\begin{equation}\label{eq:i}
0\vee i=\tfrac12(t+i),\qquad 0\vee(-i)=\tfrac12(t-i).
\end{equation}
\end{theorem}

\begin{proof}
Use seed (i) of Lemma~\ref{lem:seeds}, or seed (ii) when the formulae \eqref{eq:i} are required. In both cases $1\in A_0$ and $\Q_{\ge0}\subseteq A_0$. With $a_0=1$, $D_0=\{1\}$, $S_0=\N_0$ and $U_0=V_0$, that lemma gives (I1)--(I4), including global reducedness at the initial stage.

We construct configurations by transfinite recursion, maintaining strictly increasing cuts and (L1)--(L2) from Proposition~\ref{prop:limit} between every pair of constructed stages. We also maintain (I5) at every noninitial stage. Choices of bases are made in the usual set-theoretic setting with the axiom of choice.

At a successor stage, Proposition~\ref{prop:succ} gives a full configuration, with $a_{\alpha+1}=h>a_\alpha$. Its new generator $h$ is greater than every previous cut, so (L1) is preserved. For an older $\beta<\alpha$, the restriction equalities give
\[
P_{\alpha+1}\cap U_\beta
=(P_{\alpha+1}\cap U_\alpha)\cap U_\beta
=A_\alpha\cap U_\beta=A_\beta.
\]
Let $x=p-q$ be a reduced decomposition with respect to $A_\beta$. By induction it is reduced globally with respect to $P_\alpha$. To verify reducedness with respect to $A_\alpha$, take any representation $x=a-b$ with $a,b\in A_\alpha$. Global reducedness gives $c\in P_\alpha$ with $a=p+c$ and $b=q+c$. Since $a,p\in U_\alpha$, we also have $c=a-p\in U_\alpha$, and hence
\[
c\in P_\alpha\cap U_\alpha=A_\alpha.
\]
Thus $(p,q)$ is reduced with respect to $A_\alpha$, and Proposition~\ref{prop:succ} preserves it with respect to $P_{\alpha+1}$. The same proposition handles $\beta=\alpha$. Order convexity passes through $U_\beta\subseteq U_\alpha\subseteq U_{\alpha+1}$, since the induced orders on the older subspaces are unchanged. Thus all parts of (L2) are preserved.

At a nonzero limit ordinal $\delta$, put $a=\sup_{\beta<\delta}a_\beta$. If $a=+\infty$, stop at this first such limit ordinal, without creating a configuration with an infinite cut. If $a<+\infty$, strict increase implies $a_\beta<a$ for every $\beta<\delta$. Thus $\bigcup_{\beta<\delta}U_\beta\subseteq L_{<a}$. Conversely, for $x\in L_{<a}$, choose a noninitial stage $1\le\beta<\delta$ with $d(x)<a_\beta$. This is possible because $\delta$ is a nonzero limit ordinal. Then (I5) at that stage gives $x\in U_\beta$. Consequently the union is $L_{<a}$, and every hypothesis of Proposition~\ref{prop:limit} has been verified. It produces the required full configuration and preserves (L1)--(L2).

If the recursion continued through every $\alpha<\omega_1$, the successor cuts $a_{\alpha+1}\in\Q$ would be distinct and strictly increasing, giving an injection of $\omega_1$ into $\Q$. This is impossible. Since the successor step is always available and every finite limit cut has just been treated, the recursion must stop at a countable nonzero limit ordinal $\lambda$ with $\sup_{\beta<\lambda}a_\beta=+\infty$.

Every $x\in L$ belongs to some noninitial $U_\beta$, because its height is finite and (I5) holds there. Hence $\bigcup_{\beta<\lambda}U_\beta=L$. Applying Theorem~\ref{thm:union}, with (I4) and (L2) as its hypotheses, shows that
\[
P=\bigcup_{\beta<\lambda}P_\beta
\]
is a lattice cone on $L$ and preserves all earlier reduced decompositions. The theorem also gives multiplicative closure, since every $P_\beta$ is a positive cone. Thus $P$ is a lattice positive cone.

The initial cone gives $1\in P$ and $\Q_{>0}\subseteq P$. By (I3), each nonzero element of $P$ is either a positive rational or belongs to a layer $H_s$ with $s\ge1$. Thus $P\cap\Q=\Q_{\ge0}$. If a positive element $x$ is not rational, its height is at least $1$, whereas $d(x^{-1})=-d(x)<0$, so $x^{-1}\notin P$. Positive rationals and their inverses do belong to $P$, proving \eqref{eq:props}.

For seed (ii), $(p,q)$ is the initial reduced decomposition of $i$. It remains reduced with respect to $P$, so Theorem~\ref{thm:reduced} gives \eqref{eq:i}.
\end{proof}

\begin{corollary}\label{cor:total}
Let $P$ be a cone constructed in Theorem~\ref{thm:main}, and let $E\subseteq L$ be a subfield. If $P\cap E$ induces a total order on $E$, then $E=\Q$. In particular the order induced by $P$ on $L$ is not total when $L\ne\Q$.
\end{corollary}

\begin{proof}
Let $e\in E$, $e\ne0$. Since the induced order is total, $e>0$ or $-e>0$ in this order; in either case one of $\pm e$ lies in $P\cap E\subseteq P$, is nonzero, and its inverse is again positive, so \eqref{eq:props} gives $\pm e\in\Q_{>0}$. Hence $e\in\Q$, and $E\subseteq\Q$; finally $E=\Q$ because $E$ is a field and contains $1\ne0$.
\end{proof}

\section{The complex field}\label{sec:C}

\subsection{A lattice order on \texorpdfstring{$\C$}{C}}\label{sub:mainC}

\begin{theorem}\label{thm:C}
The field $\C$ carries a lattice positive cone $P$ with
\begin{equation}\label{eq:Cprops}
1\in P,\qquad \Q_{>0}\subseteq P,\qquad P\cap\Q=\Q_{\ge0},\qquad
\{x\in P\setminus\{0\}:\ x^{-1}\in P\}=\Q_{>0},
\end{equation}
and, for $i$ a square root of $-1$,
\begin{equation}\label{eq:Ci}
0\vee i=\tfrac12(t+i),\qquad 0\vee(-i)=\tfrac12(t-i),
\end{equation}
where $t$ is the element of the transcendence basis used in the construction. In particular $\C$ is a lattice-ordered field, and the Birkhoff--Pierce question has a positive answer.
\end{theorem}

\begin{proof}
Fix a transcendence basis $T$ of $\C$ over $\Q$ and $t\in T$, and put $F_0=\Q(T\setminus\{t\})$. By Lemma~\ref{lem:exist}, $\C$ satisfies (H1)--(H3) with $v$ trivial on $F_0^{\times}$ and $v(t)=-1$; moreover $\C$ contains $i$ with $i^{2}=-1$. Theorem~\ref{thm:main} applied to $L=\C$ with the seed of Lemma~\ref{lem:seeds}(ii) now yields a lattice positive cone $P$ with the properties \eqref{eq:props} and \eqref{eq:i}, which are \eqref{eq:Cprops} and \eqref{eq:Ci}. This answers Question~\ref{q:bp} affirmatively.
\end{proof}

The construction distinguishes a whole-field positive cone $P_\alpha$ from its lattice restriction $A_\alpha$ on $U_\alpha$. The inclusions are
\[
U_\alpha\subseteq L_{\le a_\alpha},\qquad
L_{<a_\alpha}\subseteq U_\alpha\quad(\alpha>0),
\]
and equality with $L_{\le a_\alpha}$ is not asserted. For example, at a successor cut $h$, the subspace is $\Q t^h\oplus L_{<h}$; it does not contain $it^h$. Indeed, for every $q\in\Q$, $i-q$ is nonzero and algebraic over $\Q$, so $d((i-q)t^h)=h$. A stage involves specified sets and bases, which can be uncountable, even though the recursion itself has countable ordinal length.

\subsection{Non-totality and further properties}\label{sub:partial}

\begin{corollary}\label{cor:partial}
Let $p=\tfrac12(t+i)$ and $q=\tfrac12(t-i)$. Then $p,q\in P\setminus\{0\}$, $p-q=i$, $p\wedge q=0$ and
\[
i\notin P,\qquad -i\notin P.
\]
Consequently $i$ and $0$ are incomparable, and $\le_{P}$ is not a total order on $\C$.
\end{corollary}

\begin{proof}
By \eqref{eq:Ci} and Theorem~\ref{thm:reduced}, the pair $(p,q)$ is the reduced decomposition of $i$, so $p=0\vee i$ and $q=0\vee(-i)$. Since $p\ne i$ and $q\ne-i$, neither $i$ nor $-i$ is positive. Finally $p\wedge q=0$ because $p$ and $q$ are the positive and negative parts of $i$.
\end{proof}

\begin{remark}\label{rem:inversion}
Every cone in the recursion, and its final union, satisfies
\[
P\setminus\{0\}\subseteq\Q_{>0}\cup\bigcup_{s\in\Q_{\ge1}}H_s.
\]
Thus every \emph{positive} element of height zero is rational; elements of height zero in the field need not be rational. If a cone satisfying this support condition were closed under inversion of nonzero positive elements, it could contain no element of positive height: the inverse has negative height. The cone would therefore be contained in $\Q_{\ge0}$, and it could not be directed on $\C$. This obstruction uses the support condition and does not depend on the lattice property. The directed inversion-closed cones of \cite{WYZZ} consequently lie outside this support framework.
\end{remark}

\begin{example}\label{ex:three}
Let $p=(t+i)/2$, $q=(t-i)/2$, and $V_0=\Q1\oplus\Q p\oplus\Q q$. The restriction equality in Theorem~\ref{thm:union}, applied to the final union, gives
\[
P\cap V_0=\Q_{\ge0}1+\Q_{\ge0}p+\Q_{\ge0}q.
\]
Moreover the initial reduced decompositions are preserved globally. Thus $i=p-q$ has positive part $p$ and negative part $q$, and $p\wedge q=0$ although $p,q>0$. This is a coordinate lattice order on $V_0$, and the lattice operations there agree with those on $\C$.

The same coordinate order gives $1\wedge p=0$ with $p>0$. Thus $1$ is not a \emph{weak order unit}: a positive element $e$ has this property if $e\wedge x=0$ for $x\ge0$ implies $x=0$. Also
\[
i^+i^-=pq=\frac{t^2+1}{4}\ne0,
\]
since $t$ is transcendental over $\Q$. This agrees with the Birkhoff--Pierce criterion \cite[Theorem 15]{BP}: in a lattice-ordered ring with positive identity, $1$ is a weak order unit if and only if $x^+x^-=0$ for every $x$. In particular, the present field is not an almost $f$-ring, whose defining condition is $a\wedge b=0\Rightarrow ab=0$.
\end{example}

\paragraph{A further explicit positive part.}
Start with seed (ii) and choose $1<h<2$ at the first successor step. Since $d(it)=1<h$ and
$it\notin U_0=\Q1+\Q t+\Q i$, its coset can be included in the chosen basis of $L_{<h}/U_0$.
Indeed, $it=a+bt+ci$ with $a,b,c\in\Q$ would imply $t=(a+ci)/(i-b)\in\Q(i)$, contrary to transcendence.
Choose $x_{j_0}=it$. The new coordinate cone $A_1$ then has independent positive directions
$e_0=t^h$ and $e_{j_0}=t^h+it$. By Lemma~\ref{lem:linext}(a), the identity
$it=e_{j_0}-e_0$ is a reduced decomposition. Proposition~\ref{prop:succ} and
Theorem~\ref{thm:main} preserve it globally, giving
\[
0\vee_P(it)=t^h+it,\qquad 0\vee_P(-it)=t^h.
\]
These formulae refer to the indicated choice of basis. Moreover, for every rational $s$ with
$1<s<h$, we have $t^s\in U_1$ but $t^s\notin A_1$: nonzero elements of $A_1$ have height
$0$, $1$ or $h$. Since the final restriction remains $P\cap U_1=A_1$, we have $t^s\notin P$.

A lattice-ordered field in which all squares are nonnegative is totally ordered; see \cite[p.~59]{BP} for the commutative case and \cite{Yang07} for division rings.

The restriction to the algebraic numbers can be computed exactly:
\[
P\cap\overline{\Q}=\Q_{\ge0},\qquad
(P\cap\overline{\Q})-(P\cap\overline{\Q})=\Q.
\]
Indeed, every nonzero algebraic number has height zero, by the observation after (H1)--(H3), and the support condition makes every positive element of height zero a positive rational. Conversely, all nonnegative rationals belong to $P$. Hence the restricted order on $\overline{\Q}$ is not directed, consistently with \cite{Ma}. Corollary~\ref{cor:total} further shows that the induced order is total on no subfield larger than $\Q$.

\subsection{Relation with earlier work}\label{sub:previous}

Schwartz and Yang established existence results for directed ring orders \cite{SY,SY23}. We now compare the structural properties of the present lattice cones with those of earlier constructions.

The Hahn-type embedding theorem of Conrad and Dauns \cite[Theorem III]{CD} assumes a finite lattice-group basis, that the special positive elements form a multiplicative group, and that the associated value group is torsion free. Its target field admits a compatible total order \cite[Proposition 2.2]{CD}, so it is formally real. There is no field embedding of $\C$ into a formally real field: the image of $i$ would have square $-1$. Consequently the present lattice-ordered field is outside the class covered by that embedding theorem.

For a prescribed total order on a subfield $F$, the condition $F_{>0}P\subseteq P$, together with $1\in P$, implies $F_{\ge0}\subseteq P$. Pointedness then gives $P\cap F=F_{\ge0}$. Conversely, a multiplicatively closed cone containing $F_{\ge0}$ is stable under multiplication by $F_{>0}$. Since the present cone contains $1$, Corollary~\ref{cor:total} excludes this compatibility for every $F\ne\Q$. In particular, none of the non-Archimedean coefficient fields prescribed in \cite{MWZ17,MWZ18,MWZ19,XZ20} can serve as such an $F$. Those classification results, including the quaternion results, concern the fixed $F$-algebra order framework and do not rule out our ring-order construction.

For the construction in \cite{WYZZ}, let $B$ be a transcendence basis of $\R$ over $\Q$, and let $\theta\in B$ be a distinguished real generator. Put $k=\Q(B\setminus\{\theta\})$ and $z=\theta^{-1}$, and let $\mathcal O$ be the integral closure of $k[z]_{(z)}$ in $\C$. Writing $\mathbb A=\overline{\Q}\cap\R$, the cone considered there is
\[
P_{\mathrm W}=\{0\}\cup\bigcup_{n\in\Z}z^{-n}(\mathbb A_{>0}+z\mathcal O).
\]
It is directed and closed under inversion of every nonzero positive element, and satisfies $P_{\mathrm W}\cap\overline{\Q}=\mathbb A_{\ge0}$. It is not a lattice cone. Both its restriction to $\overline{\Q}$ and its inversion property differ from those of the cone constructed here. For example, our seed (ii) gives $t=p+q\in P$, whereas $t^{-1}\notin P$ by its negative height.

The orders constructed here are not isomorphic, as ordered fields, to those arising from the ultrapower construction in \cite[Propositions~7.3 and~8.3]{He26}. Let $K$ and $\mathcal U$ be the countable algebraically closed field and the nonprincipal ultrafilter used there. The ultrapower $K^{\N}/\mathcal U$ contains the subfield $\Q^{\N}/\mathcal U$, and the induced order on this subfield is total. For each $r\in\Q$, the set of positive integers $n$ with $n>r$ is cofinite. Thus $u=[(n)_{n\ge1}]$ is greater than every element of the diagonal copy of $\Q$, so $\Q^{\N}/\mathcal U$ properly contains that copy. Transporting the order through the field isomorphism $\Phi:K^{\N}/\mathcal U\to\C$ therefore gives a totally ordered subfield $F=\Phi(\Q^{\N}/\mathcal U)$ properly containing $\Q$. If a field automorphism $\sigma$ of $\C$ carried one of our cones onto a cone obtained by that construction, then $\sigma^{-1}(F)$ would be a totally ordered subfield properly containing $\Q$, since $\sigma$ fixes $\Q$. This contradicts Corollary~\ref{cor:total}.

Theorem~\ref{thm:main} uses only (H1)--(H3); the element $i$ is needed only for the specified decomposition of $i$. The construction also applies to $\R$. Choose a transcendence basis of $\R$ over $\Q$ containing a positive element $t$, extend the valuation at infinity to the algebraic extension $\R/\Q(T)$, and take $t^r$ to be its usual positive real rational powers. The polynomial-relation argument in Lemma~\ref{lem:exist} embeds the value group in $\Q$, while these powers give every rational value and satisfy (H3). Seed (i) then gives a non-total lattice positive cone on $\R$, recovering the existence part of Wilson's theorem \cite{W} by this method.

\begin{corollary}\label{cor:many}
There are exactly $2^{\mathbf c}$ lattice positive cones on $\C$, where $\mathbf c=|\C|$. Already with a fixed transcendence basis $T$, distinguished element $t\in T$, valuation and compatible powers as in Lemma~\ref{lem:exist}, this many cones can be constructed satisfying \eqref{eq:Cprops} and \eqref{eq:Ci}. The count is of distinct subsets of $\C$, not of isomorphism classes of ordered fields.
\end{corollary}

\begin{proof}
Use seed (ii), and fix a rational $h$ with $1<h<2$ for the first successor step. Choose representatives $(x_j)_{j\in J}$ of a basis of $L_{<h}/U_0$, where $L=\C$ and $U_0=\Q1\oplus\Q p\oplus\Q q$. The set $T\setminus\{t\}$ has cardinality $\mathbf c$, consists of elements of height zero, and is linearly independent over $\Q$. Thus $L_{<h}$ has dimension $\mathbf c$ over $\Q$. Quotienting by the finite-dimensional subspace $U_0$ leaves dimension $\mathbf c$, so $|J|=\mathbf c$.

For each sign family $\varepsilon=(\varepsilon_j)_{j\in J}\in\{1,-1\}^{J}$, use the basis representatives $\varepsilon_jx_j$ in Proposition~\ref{prop:succ}. The new coordinate vectors are
\[
e_0=t^h,\qquad e_j=t^h+\varepsilon_jx_j.
\]
Since $x_j=\varepsilon_j(e_j-e_0)$, Lemma~\ref{lem:linext}(a) gives
\[
x_j^+=
\begin{cases}
t^h+x_j,&\varepsilon_j=1,\\
t^h,&\varepsilon_j=-1.
\end{cases}
\]
Continue from each such full configuration by the recursion of Theorem~\ref{thm:main}. Its proof uses only the configuration and preservation properties at the first successor stage, so every one of these choices extends to a final lattice positive cone $P_\varepsilon$ with \eqref{eq:Cprops} and \eqref{eq:Ci}. The displayed positive parts are preserved globally at all later stages. Since $x_j\ne0$, different sign families give different positive-part maps, hence different cones. There are $2^{|J|}=2^{\mathbf c}$ such families. The reverse bound follows because every cone is a subset of $\C$.
\end{proof}

Choosing any prescribed transcendental real number as $t$, for example $t=\pi$, gives the specified formula $0\vee i=(\pi+i)/2$. The resulting cone still depends on the remaining construction choices.

The maximal totally ordered subfield and the multiplication operators preserving disjointness are studied by Schwartz \cite{Sch86} and Redfield \cite{Red01}. For the present cone, Corollary~\ref{cor:total} identifies the unique maximal totally ordered subfield as $\Q$. The following proposition computes the $d$- and $f$-elements separately; for these standard terms, see \cite{Ste}.

\begin{proposition}\label{prop:multipliers}
Let $P$ be a cone supplied by Theorem~\ref{thm:C}. Define its sets of $d$-elements and $f$-elements by
\[
\begin{aligned}
\mathcal D(P)&=\{c\in P: ca\wedge cb=0\text{ whenever }a\wedge b=0\},\\
\mathcal F(P)&=\{c\in P: ca\wedge b=0\text{ whenever }a\wedge b=0\},
\end{aligned}
\]
where $a,b$ range over $\C$ and the meets are taken with respect to $P$. Then
\[
\mathcal D(P)=\mathcal F(P)=\Q_{\ge0}.
\]
\end{proposition}

\begin{proof}
Let $0\ne c\in\mathcal D(P)$ and set $M_c(x)=cx$. For every $x$, the positive elements $cx^+$ and $cx^-$ are disjoint and have difference $cx$. Theorem~\ref{thm:reduced} therefore gives
\[
M_c(x^+)=(cx)^+.
\]
Since $M_c$ is additive, \eqref{eq:supinf} shows that it preserves joins and meets. It is also bijective, so its inverse is order preserving. As $1>0$, this gives $c^{-1}=M_c^{-1}(1)>0$. The inverse condition in \eqref{eq:Cprops} now implies $c\in\Q_{>0}$. Thus $\mathcal D(P)\subseteq\Q_{\ge0}$.

If $c\in\mathcal F(P)$ and $a\wedge b=0$, first $ca\wedge b=0$, and then, applying the same condition to $b$ and $ca$, we obtain $cb\wedge ca=0$. Hence $\mathcal F(P)\subseteq\mathcal D(P)$.

Finally let $r\in\Q_{>0}$ and $a\wedge b=0$. Multiplication by any positive rational is an order automorphism, because $P$ is rationally closed. Choose an integer $N\ge\max\{1,r\}$. Since $a,b\ge0$,
\[
0\le ra\wedge b\le Na\wedge Nb=N(a\wedge b)=0.
\]
Thus $r\in\mathcal F(P)$, and $0\in\mathcal F(P)$ is immediate. Combining the inclusions proves the assertion.
\end{proof}

\subsection{Definability and classification questions}\label{sub:final}

We first consider definability of directed positive cones in the pure field language.

\begin{proposition}\label{prop:nondefinable}
No directed positive cone on $\C$ is definable, with parameters, in the pure field structure $(\C;+,-,\cdot,0,1)$.
\end{proposition}

\begin{proof}
Quantifier elimination for algebraically closed fields implies that a definable subset of the one-dimensional field sort, with parameters, is finite or cofinite; see \cite[Theorem 3.13 and Example 4.8]{Deloro}. A directed positive cone $P$ is infinite: directedness gives a nonzero $p\in P$, and its positive integer multiples are distinct elements of $P$. It cannot be cofinite either, because $P$ and $-P$ would then have an infinite intersection, contrary to $P\cap(-P)=\{0\}$. Thus $P$ is not definable in that language.
\end{proof}

Proposition~\ref{prop:nondefinable} concerns the pure field language; it does not rule out definitions in expanded structures. Corollary~\ref{cor:many} counts cones as subsets of $\C$, not ordered fields up to isomorphism. It remains to determine when the lattice cones obtained here define isomorphic ordered fields.

\section*{AI Use Disclosure}
\addcontentsline{toc}{section}{AI Use Disclosure}

\par\noindent\hspace*{2em}By \hyperref[cor:lattice-crit]{Corollary~\ref*{cor:lattice-crit}}, a pointed cone induces a lattice order exactly when every element admits a reduced decomposition. This motivates extending the cone in stages to obtain such decompositions throughout the field, while maintaining pointedness, directedness and multiplicative closure and ensuring that all previously established decompositions remain reduced.
\par

GPT-6-astra played an important role in exploring approaches to extending lattice-ordered subspaces while preserving reduced decompositions, leading to the use of coordinate and lexicographic extensions, and assisted in checking the proof of their preservation properties (\hyperref[lem:linext]{Lemma~\ref*{lem:linext}}). It also assisted in developing the choice of $t^h$ and $t^h+x_j$ as new independent positive coordinate directions in \hyperref[prop:succ]{Proposition~\ref*{prop:succ}}, and in checking the recursion arguments, particularly at limit stages (\hyperref[prop:limit]{Proposition~\ref*{prop:limit}} and the proof of \hyperref[thm:main]{Theorem~\ref*{thm:main}}). It also assisted with language polishing and copyediting during manuscript preparation. The authors take full responsibility for the paper's final content.

\begin{flushleft}
\small
\textit{Hanzhen Bao, KuangShi College, Shanghai University of
Finance and Economics, Shanghai, China.}\par
\textit{E-mail:} \href{mailto:q3556593474q@gmail.com}{\nolinkurl{q3556593474q@gmail.com}}

\textit{Yuehui Zhang, School of Mathematical Sciences, Shanghai Jiao Tong University,
800 Dongchuan Road, Shanghai 200240, China.}\par
\textit{E-mail:} \href{mailto:zyh@sjtu.edu.cn}{\nolinkurl{zyh@sjtu.edu.cn}}
\end{flushleft}
\end{document}